\documentclass[reqno,12pt]{amsart}

\usepackage{epsf}
\usepackage{graphics}
\usepackage{amssymb}
\usepackage{amsmath}

\date{}

\theoremstyle{plain}
\newtheorem{theorem}{Theorem}
\newtheorem{corollary}{Corollary}

\newtheorem{lemma}{Lemma}
\newtheorem*{question}{Question}

\theoremstyle{definition}

\theoremstyle{remark}

\def\N{{\mathbb N}}
\def\Z{{\mathbb Z}}

\def\R{{\mathbb R}}

\title{On the stable 4-genus of knots with indefinite Seifert form} 

\author{Sebastian Baader}

\begin{document}

\begin{abstract} Under a simple assumption on Seifert surfaces, we characterise knots whose stable topological 4-genus coincides with the genus.
\end{abstract}

\maketitle

\section{Introduction}

The topological 4-genus $g_4(K)$ of a knot $K$ is the minimal genus of a topological, locally flat surface embedded in the 4-ball with boundary~$K$. A well-known theorem due to Freedman asserts that knots with trivial Alexander polynomial bound a locally flat disc in the 4-ball~\cite{F}. Unlike for the classical genus $g$, there is no known algorithm that determines the topological 4-genus of a knot. The signature bound by Kauffman and Taylor~\cite{KT}, $|\sigma(K)| \leq 2g_4(K)$, fails to be sharp for the simplest knots, such as the figure-eight knot. As we will see, the signature bound becomes much more effective when the topological 4-genus is replaced by its stable version $\widehat{g}_4$ defined by Livingston~\cite{Li}:
$$ \widehat{g}_4(K)=\lim_{n \to \infty} \frac{1}{n} g_4(K^n).$$
Here $K^n$ denotes the $n$-times iterated connected sum of $K$. The existence of $\widehat{g}_4$ follows from general principles on subadditive functions (see Theorem~1 in~\cite{Li}).

\begin{theorem} Let $\Sigma \subset \R^3$ be a minimal genus Seifert surface for a knot~$K$. Assume that $\Sigma$ contains an embedded annulus with framing $+1$ or $-1$. Then the following are equivalent:

\begin{enumerate}
\item[(i)] $\widehat{g}_4(K)=g(K)$,

\smallskip
\item[(ii)] $|\sigma(K)|=2g(K)$.
\end{enumerate}
\end{theorem}

\begin{corollary} Let $\Sigma \subset \R^3$ be a minimal genus Seifert surface for a knot~$K$. If $\Sigma$ contains two embedded annuli with framings $+1$ and $-1$, then $\widehat{g}_4(K)<g(K)$.
\end{corollary}

The second condition of Theorem~1 clearly implies the first one, by the following chain of (in)equalities:
$$n2g(K)=n|\sigma(K)|=|\sigma(K^n)| \leq 2g_4(K^n) \leq 2g(K^n)=n2g(K).$$
We do not know whether the reverse implication holds without any additional assumption on Seifert surfaces.

\begin{question} Does there exist a knot $K$ with $|\sigma(K)|<2g(K)$ and $\widehat{g}_4(K)=g(K)$?
\end{question}

We conclude the introduction with an application concerning positive braid knots, i.e. knots which are closures of a positive braids. As shown in~\cite{Ba}, the only positive braid knots with $|\sigma(K)|=2g(K)$ are torus knots of type $T(2,n)$ ($n \in \N$), $T(3,4)$ and $T(3,5)$. Moreover, positive braid knots have a canonical Seifert surface (in fact, a fibre surface), which always contains a Hopf band with framing $+1$.

\begin{corollary} Let $K$ be a positive braid knot. Then $\widehat{g}_4(K)=g(K)$, if and only if $K$ is a torus knot of type $T(2,n)$ $(n \in \N)$, $T(3,4)$ or $T(3,5)$.
\end{corollary}

\section*{Acknowledgements} I would like to thank Livio Liechti for fruitful discussions, in particular for helping me get the assumption of Theorem~1 right.

\section{Constructing tori with slice boundary}

Let $K \subset S^3$ be a knot with minimal genus Seifert surface $\Sigma$. The Seifert form $V: H_1(\Sigma, \Z) \times H_1(\Sigma, \Z) \to \Z$ is defined by linear extension of the formula
$$V([x],[y])=\text{lk}(x,y^+),$$
valid for simple closed curves $x,y \subset \Sigma$. Here lk denotes the linking number and $y^+$ is a push-off of the curve $y$ in the positive direction with respect to a fixed orientation of $\Sigma$. The number $V([x],[x]) \in \Z$ is called framing of the curve $x$. The signature $\sigma(K)$ of $K$ is defined as the number of positive eigenvalues minus the number of negative eigenvalues of the symmetrised Seifert form $V+V^T$. The Alexander polynomial of $K$ is defined as $\Delta_K(t)=\det(\sqrt{t}V-\frac{1}{\sqrt{t}}V^T)$. Throughout this section, we will assume that
\begin{enumerate}
\item[(i)] the symmetrised Seifert form on $H_1(\Sigma, \Z)$ is indefinite, i.e. $$|\sigma(K)|<2g(K),$$
\item[(ii)] the surface $\Sigma$ contains an embedded annulus $A$ with framing $+1$ (the case of framing $-1$ can be reduced to this by taking the mirror image of $\Sigma$).
\end{enumerate}
Let $\Sigma^n$ be the Seifert surface for $K^n$ obtained by $n$-times iterated boundary connected sum of $\Sigma$. We define

\smallskip
\noindent
$\mathcal{F}(\Sigma)=\{m \in \Z \, |$ there exist a number $n \in \N$ and an embedded annulus $A \subset \Sigma^n$ with framing $m\}$.

\begin{lemma} $\mathcal{F}(\Sigma)=\Z$.
\end{lemma}

\begin{proof} We first show that $\Sigma$ contains an embedded annulus with negative framing. The symmetrised Seifert form $q=V+V^T$ being indefinite and non-degenerate (the latter is true for all Seifert surfaces with one boundary component), there exists a vector $\alpha \in H_1(\Sigma,\R)$ with $q(\alpha)<0$. Since negative vectors for $q$ form an open cone in $H_1(\Sigma,\R)$, there exists a simple closed curve $c \subset \Sigma$ with negative framing, i.e. $q([c])<0$. Indeed, the surface $\Sigma$ can be seen as a boundary connected sum of $g(\Sigma)$ tori; a suitable connected sum of torus knots will do.

Let $n=|q([c])|$ be the absolute value of the framing of the annulus $C \subset \Sigma$ defined by the curve $c$. We claim that $\Sigma^n$ contains an embedded annulus with framing $-1$. This can be seen by taking a split union of $C$ and $n-1$ copies of $A$ in $\Sigma^n$ (one annulus per factor), and constructing an annulus that runs through all of these once. Here we need to choose $n-1$ disjoint intervals connecting pairs of successive annuli, along which the new annulus will run back and forth, as sketched in Figure~1. In the same way, we may construct annuli with arbitrary framings.
\begin{figure}[ht]
\scalebox{1.0}{\raisebox{-0pt}{$\vcenter{\hbox{\epsffile{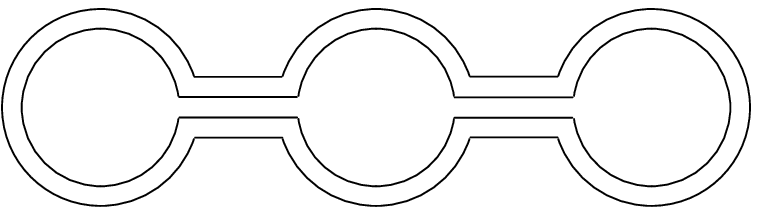}}}$}}
\caption{}
\end{figure}
\end{proof}

\begin{lemma} There exists a number $N \in \N$ and an embedded torus $T \subset \Sigma^N$ with one boundary component whose Seifert form is
$\begin{pmatrix}
0 & \pm 1 \\
0 & 0
\end{pmatrix}$, with respect to a suitable basis of $H_1(T,\Z)$. In particular, the boundary knot $L=\partial T$ has trivial Alexander polynomial.
\end{lemma}

\begin{proof} By the second assumption, $\Sigma$ contains an embedded annulus $A$ with framing $+1$. We claim that the core curve~$a$ of the annulus~$A$ is non-separating. Indeed, if the curve~$a$ was separating, it would bound a surface on one side (since the boundary of $\Sigma$ is connected), so the framing of $A$ would be zero.
As a consequence of the non-separation property of~$a$, there exists an embedded annulus $D \subset \Sigma$ which intersects $A$ in a square. The union of $A$ and $D$ is an embedded torus $T \subset \Sigma$ with one boundary component. Let
$\begin{pmatrix}
1 & b \\
c & d
\end{pmatrix}$
be the matrix representing the Seifert form on $H_1(T,\Z)$ with respect to a pair of oriented core curves of $A$ and $D$. By adding a suitable number of copies of $A$ or $B$ to $D$ in a power $\Sigma^n$, far away from the initial annulus $A \subset \Sigma$, we may impose the framing of $D$ to be $-1$, without changing its linking number with the annulus $A$. Thus we obtain an embedded torus $T' \subset \Sigma^n$ with Seifert form
$\begin{pmatrix}
1 & b \\
c & -1
\end{pmatrix}$.
An elementary base change yields
$$\begin{pmatrix}
1 & 0 \\
-c & 1
\end{pmatrix}
\begin{pmatrix}
1 & b \\
c & -1
\end{pmatrix}
\begin{pmatrix}
1 & -c \\
0 & 1
\end{pmatrix}=
\begin{pmatrix}
1 & b-c \\
0 & -bc-1
\end{pmatrix}.$$
In turn, if we replace the annulus $D$ by an annulus with $-c$ additional twists around $A$, we obtain an embedded torus $T'' \subset \Sigma^n$ with Seifert form
$\begin{pmatrix}
1 & b-c \\
0 & -bc-1
\end{pmatrix}$.
As before, we may change the individual framings of $A$ and $D$ to be zero in an even larger power $\Sigma^N$. The resulting torus, which we again denote $T \subset \Sigma^N$, has Seifert form
$V=\begin{pmatrix}
0 & b-c \\
0 & 0
\end{pmatrix}$.
We claim that $b-c=\pm 1$. Indeed, let $L=\partial T$ be the boundary knot of $T$. The Alexander polynomial of $L$ can be computed as
$$\Delta_L(t)=\det(\sqrt{t}V-\frac{1}{\sqrt{t}}V^T)=
\begin{vmatrix}
0 & \sqrt{t}(b-c) \\
-\frac{1}{\sqrt{t}}(b-c) & 0
\end{vmatrix}=(b-c)^2.$$
Since $\Delta_L(1)=1$, for all knots $L$, we conclude $b-c=\pm 1$ and
$$\Delta_L(t)=1.$$ 
\end{proof}

In order to prove Theorem~1, we need to invoke Freedman's result (\cite{F}, see also~\cite{FQ} and~\cite{GT}): knots with trivial Alexander polynomial are topologically slice.

\begin{proof}[Proof of Theorem~1]
As mentioned in the introduction, the condition $|\sigma(K)|=2g(K)$ implies $\widehat{g}_4(K)=g(K)$, without any assumption on the Seifert surface $\Sigma$. For the reverse implication, we assume $|\sigma(K)|<2g(K)$ and prove $\widehat{g}_4(K)<g(K)$. By Lemma~2, there exists a number $N \in \N$ and an embedded torus $T \subset \Sigma^N$ with one boundary component $L=\partial T$ and $\Delta_L(t)=1$. According to Freedman, there exists a topological, locally flat disc $D$ embedded in the 4-ball with boundary~$L$. We may assume that the interior of $D$ is contained in the interior of the 4-ball. Now the union of $D$ and $\Sigma^N \setminus T$ is a topological, locally flat surface embedded in the 4-ball with boundary $K^N$ and genus $Ng(K)-1$. Therefore,
$$\widehat{g}_4(K) \leq g(K)-\frac{1}{N}<g(K).$$ 
\end{proof}

\bigskip
\noindent
Universit\"at Bern, Sidlerstrasse 5, CH-3012 Bern, Switzerland

\bigskip
\noindent
\texttt{sebastian.baader@math.unibe.ch}

\end{document}